\documentclass[12pt]{article}
\usepackage{amssymb}
\usepackage{a4wide}
\usepackage{epsfig}
\usepackage{oldgerm}
\usepackage{amsmath}
\usepackage{cite}
\usepackage{amsthm}
\usepackage{fixmath}
\usepackage{amstext,color}
\usepackage{array,multirow}
\usepackage{fullpage}
\usepackage{graphicx,graphics}

\usepackage[vlined,ruled,linesnumbered]{algorithm2e}

\newtheorem{theorem}{Theorem}[section]

\newtheorem{lemma}{Lemma}[section]
\newtheorem{remark}{Remark}[section]
\newtheorem{example}{Example}
\def\CC{{\textmd \kern.24em \vrule width.02em height1.4ex depth-.05ex \kern-.26emC}}
\def\TagOnRight
\def\QQ{\rlap {\raise 0.4ex \hbox{$\scriptscriptstyle |$}} {\hskip -0.1em Q}}
\catcode`\@=11

\begin{document}
\begin{center}
 {{\bf \large {\rm {\bf Fractional Crank-Nicolson Galerkin finite element analysis for coupled time-fractional nonlocal parabolic problem }}}}
\end{center}
\begin{center}
	{\textmd {\bf Pari J. Kundaliya}}\footnote{\it Department of Mathematics,  Institute of Infrastructure, Technology, Research and Management, Ahmedabad, Gujarat, India, (pariben.kundaliya.pm@iitram.ac.in)}
%
\end{center}
\begin{abstract} In this article we propose a scheme for solving the coupled time-fractional nonlocal diffusion problem. The scheme consist of fractional Crank-Nicolson method with Galerkin finite element method (FEM) and Newton's method. We derive \emph{a priori} error estimates for fully-discrete solutions in $L^2$ and $H^1_0$ norms. Results based on the usual finite element method are provided to confirm the theoretical estimates.
\end{abstract}
{\bf AMS(MOS):} 65M12, 65M60, 35R11
\section{Introduction}
Many real world problems are modeled more accurately by fractional differential equations than integer order differential equations \cite{[r5new]}. Fractional order partial differential equations (PDEs) have numerous applications in the fields of science and engineering, such as physics, chemistry, biology and finance \cite{[Podlubny],[Anatoly],[Bruce],[LI]}. In the literature, several numerical methods are available for time discretization of time fractional PDEs. Mainly, these methods can be classified as convolution quadrature (CQ) and $L1$ type schemes \cite{[hr12],[r02],[AAl2],[r11],[r8],[r5a],[cl1],[cl2],[mk]}. There are many real-life problems which have more than one unknown function. As a motivation, authors in \cite{[C.A.]} considered a mathematical model related to two species rabbits and foxes in an island. In this model, the rate of change of the population of one species is depend on the rate of change of the population of other species. In this article, we consider the following coupled time-fractional nonlocal diffusion equation with unknowns $u$ and $v.$\\ 
Find $u$ and $v$ such that
\begin{subequations}\label{3e4}
	\begin{align}
	\label{cuc3:1.1}
	^C{D}^{\alpha}_{t}u(x,t)-M_1(l(u),l(v))\Delta{u(x,t)} & = f_{1}(u,v) &  & \mbox{in}  \: \: \Omega\times(0,T],&\\
	\label{cuc3:1.2}
	^C{D}^{\alpha}_{t}v(x,t)-M_2(l(u),l(v))\Delta{v(x,t)} & = f_{2}(u,v) &  & \mbox{in} \: \: \Omega\times(0,T],&\\
	\label{cuc3:1.3}
	u(x,t) \, = \, v(x,t) & = 0 &  & \mbox{on} \: \: \partial{\Omega}\times(0,T],&\\
	\label{cuc3:1.4}
	u(x,0) \, = \, v(x,0) & = 0 & & \mbox{in} \: \: \Omega,&
	\end{align}
\end{subequations}
where ${\Omega}$ is a bounded domain in $\mathbb{R}^d$ $(d=1,2,3)$ with smooth boundary $\partial \Omega$, $ f_{1}(u,v)$ and $ f_{2}(u,v)$ represent the forcing terms,  $l(u) = \int_{\Omega} u \, dx$, $l(v) = \int_{\Omega} v \, dx,$ $^C{D}^{\alpha}_{t}$ represents the $\alpha ^{th}$ order Caputo fractional derivative \cite{[r1]}. Above problem \eqref{cuc3:1.1}-\eqref{cuc3:1.4} can be seen as a generalization of the integer order parabolic problem given in \cite{[me],[CNS]} to fractional order. Problems of this kind have applications in biology, where $u, \, v$ can represent the densities of two populations which are interacting through the nonlocal functions $M_1$ and $M_2$  \cite{[Alm],[JCMD]}. In recent years, many researchers have paid attention to the study of nonlocal PDEs \cite{[me],[sk],[ch],[r6new],[CNS],[sp1]}. In \cite{[me],[CNS]}, author solved coupled nonlocal parabolic problem using FEM. Also, for time discretization, the Backward Euler scheme \cite{[me]} and the Crank-Nicolson scheme \cite{[CNS]} are used. In \cite{[CNCD]}, coupled time fractional nonlinear diffusion system is solved using Galerkin finite element scheme along with fractional Crank-Nicolson method.
In \cite{[mn3]}, authors used FEM with the $L1$ scheme on uniform mesh to find the numerical solution of the time-fractional nonlocal PDE
\begin{subequations}
\begin{align}
^{C}{D}^{\alpha}_{t}u(x,t)-\nabla \Big( a \Big(\int_{\Omega} u \, dx\Big) \nabla u(x,t) \Big) & = f(x,t) &  & \mbox{in}  \: \:  \Omega \times (0,T],& \nonumber\\
u(x,t) & = 0 & & \mbox{on}  \: \: \partial \Omega \times (0,T],& \nonumber \\
u(x,0) & = u_0 & & \mbox{in} \: \: \Omega.& \nonumber 
\end{align}
\end{subequations}

This scheme gives $O(\tau^{2- \alpha})$ (where $\tau$ denotes the time step size) convergence in time. In this article, we use the Fractional Crank-Nicolson method on uniform mesh to discretize the Caputo fractional derivative which gives $O(\tau^2)$ convergence in temporal direction. This method is first proposed by Dimitrov \cite{[yd]} under compatibility conditions $u \in C^3[0,T]$ and $u(0)=u_t(0)=u_{tt}(0)=0$. Here also we assume the same conditions for unknown functions $u$ and $v$. To discretize the space variable, we use the finite element method with linear bases. To handle the nonlocal term and nonlinearity, Newton's method has been used.\\
Throughout the paper, $C>0$ denotes the generic constant independent of mesh parameters $h$ and $\tau.$ Let $(\cdot,\cdot)$ denotes the inner product and $ \|\cdot\|$ denotes the norm on space $L^2(\Omega).$  For $m \in \mathbb{N}$, $H^m(\Omega)$ represents the standard Sobolev space with the norm $\|\cdot\|_m$ and $H^1_0(\Omega) := \Big\{ w \in H^1(\Omega) \, : w \, = \, 0 \: \mbox{on} \: \partial \Omega   \Big\}$.\\  
The rest of the paper is organized as follows: In section 2, we write the weak formulation and fully-discrete formulation for the given problem \eqref{3e4}. We also derive a numerical scheme to solve \eqref{3e4}. The {\it{a priori}} bound and {\it{a priori}} error estimate for fully-discrete solutions are derived in section 3 and 4, respectively. Finally, the numerical experiments given in section 5 validate the theoretical findings. \\
\noindent For existence and uniqueness results as well as numerical analysis, we make the following hypotheses on the given problem data.
\begin{itemize}
	\item {H1:} $M_i:\mathbb{R}^2\rightarrow \mathbb{R}$ is bounded with $0<m_1\leq M_i(x,y)\leq m_2, \: x, \, y \in \mathbb{R}, \: i=1,2.$
	\item {H2:} $M_i:\mathbb{R}^2\rightarrow \mathbb{R}$ is Lipschitz continuous with Lipschitz constants $L'_i,K'_i>0.$ $$|M_i(x_1,y_1)-M_i(x_2,y_2)|\leq L'_i|x_1-x_2|+K'_i|y_1-y_2|,\quad x_1, \, x_2, \, y_1, \, y_2 \in \mathbb{R}, \quad i=1,2.$$
	\item {H3:} $f_{i}:\mathbb{R}^2 \rightarrow\mathbb{R}$ is Lipschitz continuous with Lipschitz constants $L_i,K_i>0.$ $$|f_i(x_1,y_1)-f_i(x_2,y_2)|\leq L_i \|x_1-x_2\|+K_i \|y_1-y_2\|, \quad x_1, \, x_2, \, y_1, \, y_2 \in \mathbb{R}, \quad i=1,2.$$
\end{itemize}  
\section{Fully-discrete formulation}
In this section, first we write the weak formulation of \eqref{3e4} and then discretize \eqref{3e4} in both space and time variable. The weak formulation of problem \eqref{3e4} is given below.\\
Find $u(\cdot,t), v(\cdot,t) \in  H^1_0(\Omega)$ for each $t \in(0,T]$ such that 
\begin{equation}\label{e3}
\begin{split}
(\, ^{C}{D}^{\alpha}_{t}u, w ) + M_1(l(u),l(v)) \, (\nabla u, \nabla w) \, =& \, \big(f_{1}(u,v), w \big), \quad  \forall w \in H^1_0(\Omega).\\
(\, ^{C}{D}^{\alpha}_{t}v, \omega ) + M_2(l(u),l(v)) \, (\nabla v, \nabla \omega) \, =& \, \big(f_{2}(u,v), \omega \big), \quad  \forall \omega \in H^1_0(\Omega).\\
u(x,0) = v(x,0) =& \, 0 \quad \mbox{in} \; \; \Omega.
\end{split}
\end{equation}
	By using Faedo-Galerkin method \cite{[mn3]}, one can prove under the hypotheses $H1$ and $H3$, the problem \eqref{3e4} has a unique solution $\{u, \, v\}$ for $0<\alpha<1.$ \\
Now, to discretize equation \eqref{cuc3:1.1}-\eqref{cuc3:1.4} in space, we use finite element method. For that let $\Omega_h$ be a quasi uniform partition of $\Omega$ into disjoint intervals in $\mathbb{R}^1$ or triangles in $\mathbb{R}^2$ with step size $h$. Consider the $M$-dimensional subspace $X_h$ of $H^{1}_{0}(\Omega)$ such that
\begin{equation*}
X_h:=\Big\{w\in C^{0}(\bar{\Omega}): w_{|{\small T_{k}}}\in P_1(T_k), \: \forall \: T_{k}\in \Omega_h \: \, \mbox{and} \: \, w=0 \: \mbox{on} \: \partial \Omega\Big\}.
\end{equation*}
Now, we partitioned the time interval $[0, T]$ into $N$ number of sub-intervals using uniform step size $\tau = \frac{T}{N}$. Let $\tau_N := \big\{ t_n : t_n = n \tau, \: \mbox{for} \; n=0, 1,..., N\big\}$ be a partition of $[0, T]$. For each $n=1,2,\dots,N$, we denote $u(t_n)$ and $v(t_n)$ by $u^n$ and $v^n$, respectively. Let $U^n \approx u^n$ and $V^n \approx v^n$. Also, we set $U^{n, \alpha} = (1-\frac{\alpha}{2}) U^n + \frac{\alpha}{2} U^{n-1}$, $V^{n, \alpha} = (1-\frac{\alpha}{2}) V^n + \frac{\alpha}{2} V^{n-1}$.\\
To approximate the Caputo fractional derivative, we use the fractional Crank-Nicolson method.
We know that for any function $w,$ if $w(0)=0$ then $^CD^{\alpha}_{t_n} w = \, ^RD^{\alpha}_{t_n} w$, where $^RD^{\alpha}_{t_n} w$ is the $\alpha^{th}$ order Riemann-Liouville fractional derivative of $w$. Authors in \cite{[cnm]} derived the following approximation to $^RD^{\alpha}_{t_{n- \frac{\alpha}{2}}}w$ which gives $O(\tau^2)$ convergence in time.\\
\begin{equation}\label{3e1}
\begin{split}
^CD^{\alpha}_{t_{n-\frac{\alpha}{2}}} w \, = \, ^RD^{\alpha}_{t_{n-\frac{\alpha}{2}}} w \, & \approx \, D^{\alpha}_{\tau} w^{n} \, := \, \tau^{-\alpha} \, \sum_{j=0}^{n} b_{n-j} \, \phi^j, \quad n=1,2,\dots,N,
\end{split}
\end{equation}
where
\begin{equation}\label{3e2}
b_{k}= (-1)^k \, \frac{\Gamma( \alpha +1)}{\Gamma(k+1) \Gamma( \alpha-k+1)}.
\end{equation}
\begin{lemma}\label{3l2}
  \cite{[cnm],[yd],[scs],[CND]} If $w\in C^3[0,T]$ and $w(0)=0, \, w_{t}(0)=0, \, w_{tt}(0)=0,$ then the error $\| \, ^CD^{\alpha}_{t_{n-\frac{\alpha}{2}}} w - D^{\alpha}_{\tau} w^n \|$  satisfies
  \begin{equation}\label{d1trunction}
  \| \, ^CD^{\alpha}_{t_{n-\frac{\alpha}{2}}} w - D^{\alpha}_{\tau} w^n\| \leq C \tau^{2}.
  \end{equation}
\end{lemma}
The fully-discrete scheme for \eqref{3e4} is: For each $n=1,2,\ldots,N$, find $U^{n}$ and $V^{n}\in X_h$ such that  
\begin{equation*}
\begin{split}
\big( \, ^CD^{\alpha}_{\tau}U^{n}, w\big) + M_1 \big(l(U^{n, \alpha}),l(V^{n, \alpha})\big) \big(\nabla U^{n, \alpha}, \nabla w \big) &= \big( f_{1} (U^{n, \alpha},V^{n, \alpha}), w \big), \quad \forall w \in X_{h},
\end{split}
\end{equation*}
\begin{equation} \label{fully discrete3}
\begin{split}
\big( \, ^CD^{\alpha}_{\tau}V^{n}, \omega \big) + M_2\big(l(U^{n, \alpha}),l(V^{n, \alpha})\big) \big( \nabla V^{n, \alpha}, \nabla \omega \big) &= \big( f_{2} (U^{n, \alpha},V^{n, \alpha}), \omega \big), \quad \forall \omega \in X_h,\\
U^0& =0,\quad V^0=0.\\
\end{split}
\end{equation}
Using the definition of $^CD^{\alpha}_{\tau},$ (\ref{fully discrete3}) can be rewrite as 
\begin{equation}\label{fully223}
\begin{split}
\tau^{- \alpha} \, b_0 (U^{n}, w ) + M_1 \big(l(U^{n, \alpha}),l(V^{n, \alpha})\big) (\nabla U^{n, \alpha}, \nabla w ) &= \big( f_{1} (U^{n, \alpha},V^{n, \alpha}), w \big)\\
&-\tau^{- \alpha} \, \sum_{j=1}^{n-1}b_{n-j} ( U^{n, \alpha}, w_h ),\\
\tau^{- \alpha} \, b_0 ( V^{n}, \omega ) + M_2 \big(l(U^{n, \alpha}),l(V^{n, \alpha})\big) (\nabla V^{n, \alpha}, \nabla \omega )&=\big( f_{2} (U^{n, \alpha}, V^{n, \alpha}), \omega \big)\\
&-\tau^{- \alpha} \, \sum_{j=1}^{n-1}b_{n-j} \big( V^{j},\omega \big).\\
\end{split}
\end{equation}
Let $\{\psi_i(x) \}_{1\leq i\leq M}$ be the $M$ dimensional basis of $X_h$ associated with nodes of $\Omega_h.$ Therefore, for $U^n, \, V^n \in X_h$, we can find some $\beta_{i}^n,\ \gamma_{i}^n \in \mathbb{R}$ such that
\begin{equation}\label{values3}
U^n = \sum_{i=1}^{M} \beta_{i}^n \psi_i, \quad   V^n = \sum_{i=1}^{M} \gamma_{i}^n \psi_i \ .
\end{equation}
Define $\beta^n$ := $[\beta_1^n, \beta_2^n,\dots,\beta_M^n]^{\prime}$ and $\gamma^n$ := $[\gamma_1^n, \gamma_2^n,\dots,\gamma_M^n]^{\prime}$.\\
Now, substituting the values of $U^n,\ V^n$ from \eqref{values3} into \eqref{fully223}, we obtain the nonlinear algebraic equations
\begin{equation}\label{newton3}
\begin{split}
G_{i}(\beta^n, \gamma^n)=& G_{i}(U^n, V^n)=0,\quad 1\leq i\leq M,\\
H_{i}(\beta^n, \gamma^n)=& H_{i}(U^n, V^n)=0,\quad 1\leq i\leq M,\\
\end{split}
\end{equation}
where
\begin{equation}
\begin{split}
G_{i}(U^n, V^n)=\tau^{- \alpha} \, b_0 ( U^{n},\psi_i )  + \tau^{- \alpha} \, \sum_{j=1}^{n-1} b_{n-j} &( U^{j},\psi_i ) - \big(f_{1} (U^{n, \alpha}, V^{n, \alpha}), \psi_i \big)\\
&+ M_1 \big(l(U^{n, \alpha}),l(V^{n, \alpha})\big) ( \nabla U^{n, \alpha}, \nabla \psi_i ),\\
H_{i}( U^n, V^n)=\tau^{- \alpha} \, b_0 ( V^{n},\psi_i )  +\tau^{- \alpha} \, \sum_{j=1}^{n-1} b_{n-j} &( V^{j},\psi_i) - \big( f_{2}(U^{n, \alpha},V^{n, \alpha}), \psi_i\big)\\
&+ M_2 \big(l(U^{n, \alpha}),l(V^{n, \alpha})\big) ( \nabla V^{n, \alpha}, \nabla \psi_i ).\\
\end{split}
\end{equation}
If we apply Newton's method to solve \eqref{newton3}, we get the Jacobian matrix $J_1$ as follows:
\begin{equation}
\begin{split}
J_1 =\left[\begin{array}{cc} A & B \\
C & D \end{array} \right] ,
\end{split}
\end{equation}
where the elements of the matrices $A,$ $B,$ $C$ and $D$ take the form
\begin{equation}\label{3e5}
\begin{split}
(A)_{ki} \,= \, \frac{\partial G_{i}}{\partial \beta_{k}^{n}} = \, &\tau^{- \alpha} \, b_0 ( \psi_k,\psi_i ) + \Big(1-\frac{\alpha}{2}\Big) \, M_1 \big(l(U^{n, \alpha}),l(V^{n, \alpha})\big) ( \nabla \psi_k , \nabla \psi_i ) \qquad \qquad \qquad \\
&+ \Big(1-\frac{\alpha}{2}\Big) \Bigg(\frac{\partial M_1 \big(l(U^{n, \alpha}),l(V^{n, \alpha})\big)} {\partial l(U^{n})}\Bigg) \Bigg(\int_{\Omega} \psi_k \, dx \Bigg) ( \nabla U^{n, \alpha}, \nabla \psi_i ) \\ 
&- \Big(1-\frac{\alpha}{2}\Big) \, \Bigg( \frac{\partial f_{1}(U^{n},V^{n})}{\partial U^{n}} \psi_k, \psi_i\Bigg), \\
\end{split}
\end{equation}
\begin{equation}\label{3e6}
\begin{split}
(B)_{ki}= \frac{\partial G_{i}}{\partial \gamma_{k}^{n}} = \, &\Big(1-\frac{\alpha}{2}\Big) \Bigg(\frac{\partial M_1 \big(l(U^{n, \alpha}),l(V^{n, \alpha})\big)}{\partial l(V^{n})}\Bigg) \Bigg( \int_{\Omega} \psi_k \, dx\Bigg) ( \nabla U^{n, \alpha}, \nabla \psi_i ) \quad \qquad \qquad \qquad \qquad \\
&- \Big(1-\frac{\alpha}{2}\Big) \Bigg( \frac{\partial f_{1}(U^{n},V^{n})}{\partial V^{n}} \psi_k, \psi_i \Bigg),\\
\end{split}
\end{equation}
\begin{equation}\label{3e7}
\begin{split}
(C)_{pi}=\frac{\partial H_{i}}{\partial \beta_{p}^{n}} = \, &\Big(1-\frac{\alpha}{2}\Big) \Bigg( \frac{\partial M_2 \big(l(U^{n, \alpha}),l(V^{n, \alpha})\big)}{\partial l(U^{n})}\Bigg) \Bigg(\int_{\Omega} \psi_p \, dx \Bigg) ( \nabla V^{n, \alpha}, \nabla \psi_i ) \quad \qquad \qquad \qquad \qquad \\
&- \Big(1-\frac{\alpha}{2}\Big) \Bigg( \frac{\partial f_{2}(U^{n},V^{n})}{\partial U^{n}}\psi_p, \psi_i\Bigg),\\
\end{split}
\end{equation}
\begin{equation}\label{3e8}
\begin{split}
(D)_{pi} \, = \, \frac{\partial H_{i}}{\partial \gamma_{p}^{n}} = \, &\tau^{- \alpha} \, b_0 ( \psi_p,\psi_i )  + M_2 \big(l(U^{n, \alpha}),l(V^{n, \alpha})\big) ( \nabla \psi_p, \nabla \psi_i ) \qquad \qquad \qquad \qquad \qquad \\
&+ \Big(1-\frac{\alpha}{2}\Big) \Bigg(\frac{\partial M_2 \big(l(U^{n, \alpha}),l(V^{n, \alpha})\big)}{\partial l(V^{n})}\Bigg) \Bigg( \int_{\Omega} \psi_p \, dx \Bigg) ( \nabla V^{n, \alpha}, \nabla \psi_i ) \\
&-\Big(1-\frac{\alpha}{2}\Big) \Bigg( \frac{\partial f_{2}(U^{n},V^{n})}{\partial V^{n}} \psi_p, \psi_i \Bigg),\\
\end{split}
\end{equation}
where  $1\leq i,\ k,\ p\leq M.$ From equations \eqref{3e5}-\eqref{3e8}, we can observe that none of the matrices $A,$ $B,$ $C,$ $D$ are sparse and therefore Jacobian matrix $J_1$ is not sparse \cite{[me],[sk],[CNS],[sp1]}. We follow the idea given in \cite{[me],[sk],[CNS],[sp1]} to overcome above issue of sparsity. This idea was first proposed in \cite{[r3]} to solve nonlocal elliptic boundary value problem. The modified problem is defined as follows: \\
    Find $d_1,d_2 \in \mathbb{R}$ and $U^n,V^n\in X_h$ such that
  \begin{equation}\label{cuf6}
   \begin{split}
   &l(U^{n, \alpha})-d_1=0, \qquad l(V^{n, \alpha})-d_2=0,\\
   &\tau^{- \alpha} \, b_0 ( U^n, \psi_i ) + M_1 (d_1,d_2) ( \nabla U^{n, \alpha}, \nabla \psi_i ) - \big( f_{1} (U^{n, \alpha},V^{n, \alpha}), \psi_i\big) \\
   & \hspace{250 pt} + \tau^{- \alpha} \, \sum_{j=0}^{n-1} b_{n-j} ( U^{j}, \psi_i ) = 0,\\
   &\tau^{- \alpha} \, b_0 ( V^n, \psi_i ) + M_2 (d_1,d_2) ( \nabla V^{n, \alpha}, \nabla \psi_i ) - \big( f_{2}(U^{n, \alpha},V^{n, \alpha}), \psi_i \big) \\
   & \hspace{250 pt}  + \tau^{- \alpha} \, \sum_{j=0}^{n-1} b_{n-j} ( V^{j}, \psi_i) =0.
   \end{split}
    \end{equation}
    Note that if $(d_1,d_2,U^n,V^n)$ is the solution of the problem \eqref{cuf6}, then $\{U^n,V^n\}$ is the solution of the problem \eqref{fully223} and the converse is also true \cite{[sp1]}.\\
Now, to solve equation \eqref{cuf6} by using Newton's method, we rewrite \eqref{cuf6} as follows:
\begin{equation}\label{3e9}
\begin{split}
G_{i}(U^n, V^n, d_1, d_2)=0,\quad 1\leq i\leq M+1,\\
H_{i}(U^n, V^n, d_1, d_2)=0,\quad 1\leq i\leq M+1,\\
\end{split}
\end{equation}
where
\begin{equation}\label{3e10}
\begin{split}
G_{i}(U^n, V^n, d_1, d_2) \,= \, &\tau^{- \alpha} \, b_0 ( U^{n},\psi_i )  + \tau^{- \alpha} \, \sum_{j=1}^{n-1} b_{n-j} ( U^{j},\psi_i ) - \big(f_{1} (U^{n, \alpha}, V^{n, \alpha}), \psi_i \big)\\
&+ M_1 (d_1, d_2) ( \nabla U^{n, \alpha}, \nabla \psi_i ), \quad \mbox{for} \; 1 \le i \le M,\\
G_{(M+1)}(U^n, V^n, d_1, d_2&) \,= \, l(U^{n, \alpha})-d_1=0,\\
H_{i}(U^n, V^n, d_1, d_2) \, = \, &\tau^{- \alpha} \, b_0 ( V^{n},\psi_i )  +\tau^{- \alpha} \, \sum_{j=1}^{n-1} b_{n-j} ( V^{j},\psi_i) - \big( f_{2}(U^{n, \alpha},V^{n, \alpha}), \psi_i\big)\\
&+ M_2 (d_1, d_2) ( \nabla V^{n, \alpha}, \nabla \psi_i ), \quad \mbox{for} \; 1 \le i \le M,\\ 
H_{(M+1)}(U^n, V^n, d_1, d_2&) \,= \, l(V^{n, \alpha})-d_2=0.
\end{split}
\end{equation}
Now, applying Newton's method to the system of equations \eqref{3e9}, we get the following matrix equation:
\begin{gather}\label{3e58}
J
\begin{bmatrix}
{\beta^n} \\ {\gamma^n} \\ \beta \\ \gamma
\end{bmatrix}
=
\begin{bmatrix}
A_1 & B_1 & C_1 & D_1 \\ A_2 & B_2 & C_2 & D_2 \\ A_3 & B_3 & C_3 & D_3 \\ A_4 & B_4 & C_4 & D_4
\end{bmatrix}
\begin{bmatrix}
{\beta^n} \\ {\gamma^n} \\ \beta \\ \gamma
\end{bmatrix}
=
\begin{bmatrix}
\bar{G} \\ \bar{H} \\ G_{(M+1)} \\ H_{(M+1)}
\end{bmatrix},
\end{gather}
where $J$ denotes the Jacobian matrix, ${\beta^n} \, = \, [ \beta^n_1, \beta^n_2,...,\, \beta^n_M]',$ ${\gamma^n} \, = \, [ \gamma^n_1, \gamma^n_2,...,\, \gamma^n_M]',$ $\bar{G} \, = \, [G_1, G_2,..., G_M ]',$ $\bar{H} \, = \, [H_1, H_2,..., H_M ]'$ and entries of matrices $A_r,$ $B_r,$ $C_r,$ $D_r,$ $(r = 1,2,3,4)$ are given below. For $1 \le i, k \le M,$
\begin{equation}\label{3e11}
\begin{split}
&(A_1)_{ik} =  \tau^{- \alpha} \, b_0 ( \psi_k,\psi_i ) + \Big(1- \frac{\alpha}{2} \Big) \Bigg\{ M_1(d_1,d_2) ( \nabla \psi_k , \nabla \psi_i ) - \Bigg( \frac{\partial f_{1}(U^{n},V^{n})}{\partial U^{n}} \psi_k, \psi_i\Bigg)\Bigg\},\\
&(B_1)_{ik} = - \Big(1- \frac{\alpha}{2} \Big) \Bigg( \frac{\partial f_{1}(U^{n},V^{n})}{\partial V^{n}} \psi_k, \psi_i \Bigg),\\
&(C_1)_{i1}= \frac{\partial M_{1}}{\partial d_1} (d_1, d_2) \, (\nabla U^{n, \alpha}, \nabla \psi_i), \qquad (D_1)_{i1} \, = \frac{\partial M_{1}}{\partial d_2} (d_1, d_2) \, (\nabla U^{n, \alpha}, \nabla \psi_i).\\
\end{split}
\end{equation}
\begin{equation}\label{3e13}
\begin{split}
&(A_2)_{ik} = - \Big(1- \frac{\alpha}{2} \Big) \Bigg( \frac{\partial f_{2}(U^{n},V^{n})}{\partial U^{n}} \psi_k, \psi_i \Bigg),\\
&(B_2)_{ik} =  \tau^{- \alpha} \, b_0 ( \psi_k,\psi_i ) + \Big(1- \frac{\alpha}{2} \Big) \Bigg\{ \! \! M_2(d_1,d_2) ( \nabla \psi_k , \nabla \psi_i ) \! - \! \Bigg( \! \frac{\partial f_{2}(U^{n},V^{n})}{\partial V^{n}} \psi_k, \psi_i\Bigg)\Bigg\},\\
&(C_2)_{i1}= \frac{\partial M_{2}}{\partial d_1} (d_1, d_2) \, (\nabla V^{n, \alpha}, \nabla \psi_i), \qquad (D_2)_{i1} \, = \frac{\partial M_{2}}{\partial d_2} (d_1, d_2) \, (\nabla V^{n, \alpha}, \nabla \psi_i).\\
\end{split}
\end{equation}
\begin{equation}\label{3e15}
\begin{split}
&(A_3)_{1k} \, = \Big(1- \frac{\alpha}{2} \Big) \int_{\Omega} \psi_k \, dx, \qquad (B_3)_{1k} \, = 0, \qquad (C_3)_{11} \, = -1, \qquad (D_3)_{11} \, = 0.\\
\end{split}
\end{equation}
\begin{equation}\label{3e16}
\begin{split}
&(A_4)_{1k} \, = 0, \qquad (B_4)_{1k} \, = \Big(1- \frac{\alpha}{2} \Big) \int_{\Omega} \psi_k \, dx, \qquad (C_4)_{11} \, = 0, \qquad (D_4)_{11} \, = -1.\\
\end{split}
\end{equation}
From equations \eqref{3e11}-\eqref{3e16}, it can be seen that $A_1,$ $B_1,$ $A_2,$ $B_2$ are sparse matrices \cite{[CNS],[CNCD],[sp1]} and hence $J$ is a sparse matrix. 
\section{\textit{A priori} bound}
In this section, we provide {\it{a priori}} bound for the fully-discrete scheme \eqref{fully discrete3}. For this we need following Lemma. 
\begin{lemma}\label{3l1}
  \cite{[CND]} For any function $ \phi(\cdot, t)$ defined on $\tau_N$, one has
	\begin{equation}
	   \frac{1}{2} \,  ^CD^{\alpha}_{\tau} \| \phi^n \|^2 \, \le \, \big( \, ^CD^{\alpha}_{\tau} \phi^{n}, \, \phi^{n, \alpha} \big), \\
	\end{equation}
	where $\phi^{n, \alpha} := (1- \frac{\alpha}{2}) \phi^{n} + \frac{\alpha}{2}\phi^{n-1},$ \ for $n = 1,2,\dots,N.$
\end{lemma}
For derivation of {\it{a priori}} bound and error estimate, we also use following Discrete fractional Gr\"onwall type inequality.
\begin{lemma}\label{Dong}
	\cite{[CND]} Suppose the nonnegative sequences $\{\omega^n , \phi^n |    n=0,1,2,\dots\}$ satisfy
	\begin{equation*}
	^CD^{\alpha}_{\tau} \omega^n \leq \lambda_1 \omega^n + \lambda_2 \omega^{n-1} + \phi^n, \quad n\geq 1,\\
	\end{equation*}
	where $\lambda_1$ and $\lambda_2$ are nonnegative constants. Then there exists a positive constant ${\tau}^\star$ such that when $\tau \leq {\tau}^\star$,
	\begin{equation*}
	\omega^n \leq 2\Big(\omega^0 + \frac{t_{n}^{\alpha}}{\Gamma(1+\alpha)} \max_{0 \leq j \leq n} \phi^{j} \Big) E_{\alpha}(2\lambda t_{n}^{\alpha}), \quad 1\leq n\leq N,
	\end{equation*}
	where $E_{\alpha}(z)$ is the Mittag-Leffler function and $\lambda=\lambda_1+\frac{\lambda_2}{2-2^{(1-\alpha)}}.$
\end{lemma}
\begin{theorem}\label{stability1 theorem3}
	Let $(U^n,V^n)$ (for $1 \le n \le N$) be the solution of $(\ref{fully discrete3})$.  Then there exists a positive constant  ${\tau}^{*}$ (independent of $h$) such that when $\tau\leq {\tau}^{*},$ $U^n, \, V^n$ satisfy
	\begin{equation}
	\|U^n\|+\|V^n\|\leq C,
	\end{equation}
		\begin{equation}
	\|\nabla U^n\|+\|\nabla V^n\|\leq C,
	\end{equation}
\end{theorem}
\begin{proof}
	For $\forall w\in X_{h},$ from (\ref{fully discrete3}), we have
	\begin{equation}\label{3e18}
	\big( \, ^CD^{\alpha}_{\tau}U^{n}, w \big)+ M_1\big(l(U^{n, \alpha}),l(V^{n, \alpha})\big) (\nabla U^{n, \alpha}, \nabla w) = \big( f_{1}(U^{n, \alpha},V^{n, \alpha}), w \big).\\
	\end{equation}
	Choosing  $w=U^{n, \alpha}$ and then using the Cauchy-Schwarz inequality along with the inequality $ab \le \frac{1}{2} a^2 + \frac{1}{2} b^2$, we can obtain
	\begin{equation} \label{stb13}
	\big( \, ^CD^{\alpha}_{\tau}U^{n}, U^{n, \alpha} \big) + M_1 \big(l(U^{n, \alpha}),l(V^{n, \alpha})\big) \|\nabla U^{n, \alpha}\|^2\leq\frac{1}{2}\big(\|f_{1}(U^{n, \alpha},V^{n, \alpha})\|^2+\| U^{n, \alpha}\|^2\big).
	\end{equation}
	Since $f_{1}$ is Lipschitz continuous, we have
	\begin{equation*}
	\begin{split}
	\Big|\| {f_{1}(U^{n, \alpha}, V^{n, \alpha})} \| - \| {f_{1}(0, 0)} \| \Big| \, &\le \| {f_{1}(U^{n, \alpha}, V^{n, \alpha})-f_{1}(0, 0)} \| \\
	&\le \, L_1 \, \|{U^{n, \alpha}} \| + K_1 \, \|{V^{n, \alpha}}\|.
	\end{split}
	\end{equation*}
	Therefore,
	\begin{equation}\label{3e17}
	\begin{split}
	 \| {f_1(U^{n, \alpha}, V^{n, \alpha})} \| \, \le \, \| {f_1(0,0)} \| +  L_1 \, \|{U^{n, \alpha}} \| + K_1 \, \|{V^{n, \alpha}}\| \, \le \,  C \, \big(1+\|{U^{n, \alpha}} \|+\|{V^{n, \alpha}} \|\big).\\
	\end{split}
	\end{equation}
	By using the bound of $M_1$ and from equation \eqref{3e17}, we can write equation \eqref{stb13} as
	\begin{equation}\label{dp3}
	\begin{split}
	\big( \, ^CD^{\alpha}_{\tau}U^{n}, U^{n, \alpha}\big)+ m_1 \,  \|\nabla U^{n, \alpha}\|^2
	&\leq C \big((1+\| U^{n, \alpha}\|+\| V^{n, \alpha}\|)^2+\| U^{n, \alpha}\|^2 \big),\\
	( \, ^CD^{\alpha}_{\tau}U^{n, \alpha}, U^{n, \alpha}) &\leq C \big( 1+\| U^{n, \alpha}\|^2+\| V^{n, \alpha}\|^2 \big).\\
	\end{split}
	\end{equation}
	An application of Lemma \ref{3l1} in \eqref{dp3}, gives
	\begin{equation} \label{stb233}
	{^CD}^{\alpha}_{\tau}\|U^{n}\|^2\leq C\big(1+\| U^{n, \alpha}\|^2+\| V^{n, \alpha}\|^2\big).\\
	\end{equation}
	Similarly, we can get the estimate for $V^{n}$ as follows:
	\begin{equation} \label{stb3}
	{^CD}^{\alpha}_{\tau}\|V^{n}\|^2\leq C\big(1+\| V^{n, \alpha}\|^2+\| U^{n, \alpha}\|^2\big).\\
	\end{equation}
	Adding \eqref{stb233} and \eqref{stb3}, we get
	\begin{equation}\label{stb3a}
	{^CD}^{\alpha}_{\tau}(\|U^{n}\|^2+\|V^{n}\|^2)\leq C\big(1+\| U^{n, \alpha}\|^2+\| V^{n, \alpha}\|^2\big).\\
	\end{equation}
	Using Lemma $\ref{Dong}$ in \eqref{stb3a}, we can arrive at
	\begin{equation} \label{stb43}
	\|U^{n}\|^2+\|V^{n}\|^2\leq C.\\
	\end{equation}
	For $a,b\geq 0,$ using $\frac{1}{2}(a+b)^2\leq a^2+b^2$ in \eqref{stb43}, we obtain
	\begin{equation*}
	\|U^{n}\|+\|V^{n}\|\leq C.\\
	\end{equation*}
	Now, we choose $w=\ ^CD^{\alpha}_{\tau}U^{n}$ and then use Cauchy-Schwarz inequality in \eqref{3e18} to obtain
	\begin{equation} \label{3e19}
	\begin{split}
	\| \, ^CD^{\alpha}_{\tau} U^{n}\|^2 + M_1 \big(l(U^{n, \alpha}),l(V^{n, \alpha})\big) \big( \nabla U^{n, \alpha}, \nabla \, ^CD^{\alpha}_{\tau}U^{n} \big) \leq & \|f_{1}(U^{n, \alpha},V^{n, \alpha})\| \| \,  ^CD^{\alpha}_{\tau}U^{n}\|.\\
	\end{split}
	\end{equation}
	We divide both sides of \eqref{3e19} by $M_1 \big(l(U^{n, \alpha}),l(V^{n, \alpha})\big)$ and then use bound of $M_1$ to get
	\begin{equation} \label{stb133}
	\frac{1}{m_2} \| \, ^CD^{\alpha}_{\tau}U^{n}\|^2+(\nabla U^{n, \alpha},\nabla \,  ^CD^{\alpha}_{\tau}U^{n}) \leq \frac{1}{m_1} \big( \|f_{1}(U^{n, \alpha},V^{n, \alpha})\|  \| ^CD^{\alpha}_{\tau} U^{n}\| \big).
	\end{equation}
	An application of inequality $ab \le \frac{\epsilon}{2} a^2 + \frac{1}{2 \epsilon} b^2$ $\big(\mbox{with} \; \epsilon = m_2\big)$ gives
	\begin{equation} \label{3e20}
	\frac{1}{m_2} \| \, ^CD^{\alpha}_{\tau}U^{n}\|^2 + (\nabla U^{n, \alpha}, \nabla \,  ^CD^{\alpha}_{\tau}U^{n}) \leq \frac{m_2}{2 m_1^2} \, \|f_{1}(U^{n, \alpha},V^{n, \alpha})\|^2 + \frac{1}{2 m_2} \| \, ^CD^{\alpha}_{\tau} U^{n}\|.
	\end{equation}
Using Lemma \ref{3l1} and \eqref{3e17} with Poincar\'e inequality in \eqref{3e20}, we can arrive at
\begin{equation}\label{dp33}
^CD^{\alpha}_{\tau} \|\nabla U^{n}\|^2
	\leq C\big(1+\|\nabla U^{n, \alpha}\|^2+\|\nabla V^{n, \alpha}\|^2\big).\\
\end{equation}
	Similarly, we get an estimate for $V^{n}$ as
	\begin{equation} \label{stb33}
	{^CD}^{\alpha}_{\tau} \|\nabla V^{n}\|^2 \leq C \big(1+\|\nabla V^{n, \alpha}\|^2+\|\nabla U^{n, \alpha}\|^2\big).\\
	\end{equation}
On adding \eqref{dp33} and \eqref{stb33}, we get
	\begin{equation}\label{3e21}
	{^CD}^{\alpha}_{\tau} (\|\nabla U^{n}\|^2+\|\nabla V^{n}\|^2) \leq C \big(1+\|\nabla U^{n, \alpha}\|^2+\|\nabla V^{n, \alpha}\|^2\big).\\
	\end{equation}
	An application of Lemma $\ref{Dong}$ to \eqref{3e21} gives required result
	\begin{equation*}
	\|\nabla U^{n}\|+\|\nabla V^{n}\|\leq C.\\
	\end{equation*}
	This completes the proof. 
\end{proof}
\begin{theorem} 
	Let $U^0,\ U^1,\dots, U^{n-1}$ and  $V^0,\ V^1,\dots, V^{n-1}$ are given. Then there exists a positive constant $\tau^{*}$ (independent of $h$) such that when $\tau\leq \tau^{*},$ the problem \eqref{fully discrete3} has a unique solution $U^n, \, V^n \in X_h,$ for all $1\leq n\leq N$.
\end{theorem}
\begin{proof} The proof follows in similar lines as given in of \cite[Theorem 1]{[CNCD]}.
\end{proof}
\section{\textit{A priori} error estimate}
In order to derive convergence estimate, we use Ritz projection operator \cite{[vth]} given by
\begin{equation}
\label{proj3}
(\nabla \phi,\nabla w)=( \nabla R_h \phi,\nabla w),\hspace{.5cm} \forall \phi \in H_0^1(\Omega),\ w \in X_h.
\end{equation}
In following Theorem, we state an approximation property for the operator $R_h$ which will be useful in the derivation of {\it{a priori}} error estimate.
\begin{theorem}\cite{[vth]}\label{projerr} There exists $C>0$, independent of $h$ such that
\begin{equation}\label{3e22}
\begin{split}
\|{\phi-R_h\phi} \|_{L^2(\Omega)} \, + h \, \|{\nabla (\phi-R_h\phi)}\|_{L^2(\Omega)} \, \le& \, Ch^2 \, \|{\Delta \phi}\|_{L^2(\Omega)}, \quad \forall \phi \in H^2(\Omega) \cap H^1_0(\Omega).\\
\end{split}
\end{equation}
\end{theorem}
Using the intermediate projection $R_h,$ we split the error into two parts as
\begin{equation*}
\begin{split}
u(x,t_n)-U^n=u^n-U^n=&(u^n-R_hu^n)+(R_hu^n-U^n)=\zeta_{1}^n+\chi_{1}^n,\\
v(x,t_n)-V^n=v^n-V^n=&(v^n-R_hv^n)+(R_hv^n-V^n)=\zeta_{2}^n+\chi_{2}^n.\\
\end{split}
\end{equation*}
\begin{theorem}\label{fully1 theorem3}
		Let  $(U^n,V^n)$ be the solution of fully-discrete scheme \eqref{fully discrete3}. Then under the sufficient regularity assumptions  on the solution $u, \ v\in C^3([0,T];L^2(\Omega))\cap C^2([0,T];H^2(\Omega) \cap H^1_0(\Omega))$, $\Big( \frac{\partial^i u}{\partial t^i}\Big)_{t=0} = \Big( \frac{\partial^i v}{\partial t^i}\Big)_{t=0} = 0,$ for $i=0,1,2$ of $(\ref{cuc3:1.1})$-$(\ref{cuc3:1.4}),$ there exists a positive constant  ${\tau}^{*}$ (independent of $h$) such that when $\tau\leq {\tau}^{*},$ the following estimates hold.
		\begin{equation}
		\begin{split}
		\|u^n-U^n\|+\|v^n-V^n\| \le& \, C \, \big(h^2 + \tau^2 \big),\\
		\|\nabla u^n - \nabla U^n\| + \| \nabla v^n - \nabla V^n\| \le& \, C \, \big(h + \tau^2 \big),
		\end{split}
		\end{equation}
		where  $n= 1,2,\dots,N.$ 
\end{theorem}
\begin{proof}
	For any $w \in X_h,$ we have
	\begin{equation}\label{3e23}
	\begin{split}
	&\big(\, ^C{D}^{\alpha}_{\tau} \chi_{1}^n, w \big) \, + \, M_1 \big(l(U^{n, \alpha}), l(V^{n, \alpha})\big) \, (\nabla \chi_{1}^{n, \alpha}, \nabla w) \qquad\\
	&\;= \, \big(\, ^C{D}^{\alpha}_{\tau} R_hu^n, w \big) + \, M_1 \big(l(U^{n, \alpha}), l(V^{n, \alpha})\big) \, (\nabla R_hu^{n, \alpha}, \nabla w) - \big(\, ^C{D}^{\alpha}_{\tau} U^n, w \big) \\
	&\qquad- \, M_1 \big(l(U^{n, \alpha}), l(V^{n, \alpha})\big) \, (\nabla U^{n, \alpha}, \nabla w).\\
	&\;= \, \big(\, ^C{D}^{\alpha}_{\tau} R_hu^n, w \big) + \, M_1 \big(l(U^{n, \alpha}), l(V^{n, \alpha})\big) \, (\nabla u^{n, \alpha}, \nabla w) - \big( f_{1}(U^{n, \alpha}, V^{n, \alpha}), w \big) \\
	& \qquad + \big( f_{1}(u^{n-\frac{\alpha}{2}}, v^{n-\frac{\alpha}{2}}), w \big) - \big( \, ^{C}{D}^{\alpha}_{t_{n-\frac{\alpha}{2}}}u, w \big) - M_1 \big(l(u^{n-\frac{\alpha}{2}}), l(v^{n-\frac{\alpha}{2}})\big) \, (\nabla u^{n-\frac{\alpha}{2}}, \nabla w) \\
	& \qquad + \, M_1 \big(l(u^{n-\frac{\alpha}{2}}), l(v^{n-\frac{\alpha}{2}})\big) \, (\nabla u^{n, \alpha}, \nabla w) -  M_1 \big(l(u^{n-\frac{\alpha}{2}}), l(v^{n-\frac{\alpha}{2}})\big) \, (\nabla u^{n, \alpha}, \nabla w)  \\
	&\;= \, \big(\, ^C{D}^{\alpha}_{\tau} R_hu^n - \, ^{C}{D}^{\alpha}_{t_{n-\frac{\alpha}{2}}}u , w \big) + M_1 \big(l(u^{n-\frac{\alpha}{2}}), l(v^{n-\frac{\alpha}{2}})\big) (\nabla u^{n, \alpha} - \nabla u^{n-\frac{\alpha}{2}}, \nabla w) \\
	& \qquad + \Big\{ M_1 \big(l(U^{n, \alpha}), l(V^{n, \alpha})\big) - M_1 \big(l(u^{n-\frac{\alpha}{2}}), l(v^{n-\frac{\alpha}{2}})\big) \Big\} (\nabla u^{n, \alpha}, \nabla w) \\
	& \qquad + \big( f_{1}(u^{n-\frac{\alpha}{2}}, v^{n-\frac{\alpha}{2}}) - f_{1}(U^{n, \alpha}, V^{n, \alpha}), w \big).
	\end{split}
	\end{equation}
	Setting $w=\chi_{1}^{n, \alpha}$ in $(\ref{3e23})$ to get
	\begin{equation}\label{3e24}
	\begin{split}
	  &\big(\, ^C{D}^{\alpha}_{\tau} \chi_{1}^n, \chi_{1}^{n, \alpha} \big) \, + \, M_1 \big(l(U^{n, \alpha}), l(V^{n, \alpha})\big) \, (\nabla \chi_{1}^{n, \alpha}, \nabla \chi_{1}^{n, \alpha}) \qquad\\
	  &\;= \, \big(\, ^C{D}^{\alpha}_{\tau} R_hu^n - \, ^{C}{D}^{\alpha}_{t_{n-\frac{\alpha}{2}}}u , \, \chi_{1}^{n, \alpha} \big) + M_1 \big(l(u^{n-\frac{\alpha}{2}}), l(v^{n-\frac{\alpha}{2}})\big) (\nabla u^{n, \alpha} - \nabla u^{n-\frac{\alpha}{2}}, \nabla \chi_{1}^{n, \alpha}) \\
	  & \qquad + \Big\{ M_1 \big(l(U^{n, \alpha}), l(V^{n, \alpha})\big) - M_1 \big(l(u^{n-\frac{\alpha}{2}}), l(v^{n-\frac{\alpha}{2}})\big) \Big\} (\nabla u^{n, \alpha}, \nabla \chi_{1}^{n, \alpha}) \\
	  & \qquad + \big( f_{1}(u^{n-\frac{\alpha}{2}}, v^{n-\frac{\alpha}{2}}) - f_{1}(U^{n, \alpha}, V^{n, \alpha}), \, \chi_{1}^{n, \alpha} \big).
	\end{split}
	\end{equation}
	Also, from the assumptions on solution $u$ and $v$, we can find $R_u, \, R_v \, >0$ such that
	\begin{equation}\label{3e27}
	\begin{split}
	\|u^{n, \alpha}  \|_{ H^2(\Omega)} \le R_u, \quad \| v^{n, \alpha} \|_{ H^2(\Omega)} \le R_v.
	\end{split}
	\end{equation}
	Using the bound of $M_1$, \eqref{3e27} and Cauchy-Schwarz inequality in \eqref{3e24}, we get
	\begin{equation}\label{3e25}
	\begin{split}
	&\big(\, ^C{D}^{\alpha}_{\tau} \chi_{1}^n, \chi_{1}^{n, \alpha} \big) \, + m_1 \, \| \nabla \chi_{1}^{n, \alpha} \|^2 \qquad\\
	&\; \le \, \| \, ^C{D}^{\alpha}_{\tau} R_hu^n - \, ^{C}{D}^{\alpha}_{t_{n-\frac{\alpha}{2}}}u \| \, \| \chi_{1}^{n, \alpha} \| + m_2 \, \| \nabla u^{n, \alpha} - \nabla u^{n-\frac{\alpha}{2}} \| \, \| \nabla \chi_{1}^{n, \alpha} \| \\
	& \qquad + R_u \, \big| M_1 \big(l(U^{n, \alpha}), l(V^{n, \alpha})\big) - M_1 \big(l(u^{n-\frac{\alpha}{2}}), l(v^{n-\frac{\alpha}{2}})\big) \big| \,  \|\nabla \chi_{1}^{n, \alpha} \| \\
	& \qquad + \big\| f_{1}(u^{n-\frac{\alpha}{2}}, v^{n-\frac{\alpha}{2}}) - f_{1}(U^{n, \alpha}, V^{n, \alpha}) \big\| \, \| \chi_{1}^{n, \alpha} \|.
	\end{split}
	\end{equation}
	Applying Poincar$\acute{e}$ inequality and the inequality $ab \le \frac{\epsilon}{2} a^2 + \frac{1}{2 \epsilon} b^2$ $\big(\mbox{with} \; \epsilon = \frac{4}{m_1} \big)$ to \eqref{3e25}, we can obtain
	\begin{equation}\label{3e26}
	\begin{split}
	&\big(\, ^C{D}^{\alpha}_{\tau} \chi_{1}^n, \chi_{1}^{n, \alpha} \big) \, + m_1 \, \| \nabla \chi_{1}^{n, \alpha} \|^2 \qquad\\
	&\; \le \frac{2}{m_1} \, \| \, ^C{D}^{\alpha}_{\tau} R_hu^n - \, ^{C}{D}^{\alpha}_{t_{n-\frac{\alpha}{2}}}u \|^2 + \frac{2}{m_1} \, \big\| f_{1}(u^{n-\frac{\alpha}{2}}, v^{n-\frac{\alpha}{2}}) - f_{1}(U^{n, \alpha}, V^{n, \alpha}) \big\|^2 \\
	& \qquad + \frac{2 R_u^2}{m_1} \, \big| M_1 \big(l(U^{n, \alpha}), l(V^{n, \alpha})\big) - M_1 \big(l(u^{n-\frac{\alpha}{2}}), l(v^{n-\frac{\alpha}{2}})\big) \big|^2 \\
	& \qquad + \frac{2 \, m_2^2}{m_1} \, \| \nabla u^{n, \alpha} - \nabla u^{n-\frac{\alpha}{2}} \|^2 + \frac{m_1}{2} \, \| \nabla \chi_{1}^{n, \alpha} \|^2.
	\end{split}
	\end{equation}
	Lipschitz continuity of functions $M_1$ and $f_1$ gives
	\begin{equation}\label{3e28}
	\begin{split}
	&\big| M_1 \big(l(U^{n, \alpha}), l(V^{n, \alpha})\big) - M_1 \big(l(u^{n-\frac{\alpha}{2}}), l(v^{n-\frac{\alpha}{2}})\big) \big|\\
	 &\le L'_1 \, \big|l(U^{n, \alpha}) - l(u^{n-\frac{\alpha}{2}}) \big| + K'_1 \, \big|l(V^{n, \alpha}) - l(v^{n-\frac{\alpha}{2}}) \big|\\
	&\le L'_1 C\, \|U^{n, \alpha} - u^{n-\frac{\alpha}{2}}\| + K'_1 C\, \|V^{n, \alpha} - v^{n-\frac{\alpha}{2}}\|\\
	&\le C \, \big\{ \| \zeta_{1}^{n, \alpha}\| + \|\chi_{1}^{n, \alpha}\| + \| \zeta_{2}^{n, \alpha}\| + \|\chi_{2}^{n, \alpha}\| + \|u^{n, \alpha} - u^{n-\frac{\alpha}{2}}\| + \|v^{n, \alpha} - v^{n-\frac{\alpha}{2}}\| \big\}
	\end{split}
	\end{equation}
	and
	\begin{equation}\label{3e29}
	\begin{split}
	&\big\| f_{1}(u^{n-\frac{\alpha}{2}}, v^{n-\frac{\alpha}{2}}) - f_{1}(U^{n, \alpha}, V^{n, \alpha}) \big\| \\
	&\le L_1 C\, \|U^{n, \alpha} - u^{n-\frac{\alpha}{2}}\| + K_1 C\, \|V^{n, \alpha} - v^{n-\frac{\alpha}{2}}\| \\
	&\le C \, \big\{ \| \zeta_{1}^{n, \alpha}\| + \|\chi_{1}^{n, \alpha}\| + \| \zeta_{2}^{n, \alpha}\| + \|\chi_{2}^{n, \alpha}\| + \|u^{n, \alpha} - u^{n-\frac{\alpha}{2}}\| + \|v^{n, \alpha} - v^{n-\frac{\alpha}{2}}\| \big\}.
	\end{split}
	\end{equation}
Thus, from equations \eqref{3e26}-\eqref{3e29} and Lemma \ref{3l1}, we have	
\begin{equation}\label{3e30}
\begin{split}
  ^C{D}^{\alpha}_{\tau} \| \chi_{1}^n \|^2 \le& \, C \, \big\{  \| \, ^C{D}^{\alpha}_{\tau} R_hu^n - \, ^{C}{D}^{\alpha}_{t_{n-\frac{\alpha}{2}}}u \|^2 + \| \nabla u^{n, \alpha} - \nabla u^{n-\frac{\alpha}{2}} \|^2 + \| \zeta_{1}^{n, \alpha}\|^2  \\
  &+ \|\chi_{1}^{n, \alpha}\|^2 + \| \zeta_{2}^{n, \alpha}\|^2 + \|\chi_{2}^{n, \alpha}\|^2 + \|u^{n, \alpha} - u^{n-\frac{\alpha}{2}}\|^2 + \|v^{n, \alpha} - v^{n-\frac{\alpha}{2}}\|^2 \big\},
\end{split}
\end{equation}
where $C$ is dependent on $m_1,$ $m_2,$ $R_u,$ $L'_1,$ $K'_1,$ $L_1,$ $K_1.$ \\
Now, it follows from Taylor's theorem that
\begin{equation}\label{3e31}
\begin{split}
  \|u^{n, \alpha} - u^{n-\frac{\alpha}{2}}\| \, \le& \, \frac{\alpha}{2} \Big( 1- \frac{\alpha}{2} \Big) \, \tau \int_{t_{n-1}}^{t_n} \| u_{tt} (\eta)\| \, d\eta \, \le \, C \, \tau^2.
\end{split}
\end{equation}
Similarly, 
\begin{equation}\label{3e32}
\begin{split}
\|v^{n, \alpha} - v^{n-\frac{\alpha}{2}}\| \, \le \, C \, \tau^2, \qquad \|\nabla u^{n, \alpha} - \nabla u^{n-\frac{\alpha}{2}}\| \, \le \, C \, \tau^2.
\end{split}
\end{equation}
From Lemma \ref{3l2} and \eqref{3e22}, we have
\begin{equation}\label{3e33}
\begin{split}
  \| \, ^C{D}^{\alpha}_{\tau} R_hu^n - \, ^{C}{D}^{\alpha}_{t_{n-\frac{\alpha}{2}}}u \| \le& \| \, ^C{D}^{\alpha}_{\tau} R_hu^n - \, ^{C}{D}^{\alpha}_{t_{n-\frac{\alpha}{2}}}R_hu \| + \| \, ^{C}{D}^{\alpha}_{t_{n-\frac{\alpha}{2}}}R_hu - \, ^{C}{D}^{\alpha}_{t_{n-\frac{\alpha}{2}}}u \| \\
  \le& \, C \, \big( \tau^2 + h^2 \big).
\end{split}
\end{equation}
Using \eqref{3e22} and \eqref{3e31}-\eqref{3e33} in \eqref{3e30}, we get
\begin{equation}\label{3e34}
\begin{split}
  ^C{D}^{\alpha}_{\tau} \| \chi_{1}^n \|^2 \le& \, C \big(\|\chi_{1}^{n, \alpha}\|^2 + \|\chi_{2}^{n, \alpha}\|^2 + \big( h^2 + \tau^2 \big)^2 \big).
\end{split}
\end{equation}
Similarly, we can get an estimate for $\|\chi_{2}^{n, \alpha}\|$ as follows:
\begin{equation}\label{3e35}
\begin{split}
^C{D}^{\alpha}_{\tau} \| \chi_{2}^n \|^2 \le& \, C \big(\|\chi_{1}^{n, \alpha}\|^2 + \|\chi_{2}^{n, \alpha}\|^2 + \big( h^2 + \tau^2 \big)^2 \big).
\end{split}
\end{equation}
Therefore, from \eqref{3e34} and \eqref{3e35}  
\begin{equation}\label{3e36}
\begin{split}
 ^C{D}^{\alpha}_{\tau} \big( \| \chi_{1}^n \|^2 + \| \chi_{2}^n \|^2 \big) \le& \, C \Big( 1 - \frac{\alpha}{2}\Big)^2 \big(\|\chi_{1}^n \|^2 + \|\chi_{2}^n \|^2 \big) + \frac{C \, \alpha^2}{4} \big(\|\chi_{1}^{n-1} \|^2 + \|\chi_{2}^{n-1} \|^2 \big) \\
 &+ C \, \big( h^2 + \tau^2 \big)^2.
\end{split}
\end{equation}
An application of Lemma $\ref{Dong}$ $\Big(\mbox{with} \; \lambda_1 = C \, \big( 1 - \frac{\alpha}{2}\big)^2,$ $\lambda_2 = \frac{C \, \alpha^2}{4},$ $\phi^n = C \, ( h^2 + \tau^2 )^2\Big)$ gives
\begin{equation}\label{3e37}
\begin{split}
  \| \chi_{1}^n \|^2 + \| \chi_{2}^n \|^2  \le& \, C \, \big( h^2 + \tau^2 \big)^2,
\end{split}
\end{equation}
where $C$ is dependent on $\lambda_1, \lambda_2, \alpha, T.$\\
Therefore, 
\begin{equation}\label{3e37a}
\begin{split}
\| \chi_{1}^n \| + \| \chi_{2}^n \| \le& \, C \, \big( h^2 + \tau^2 \big).
\end{split}
\end{equation}
Finally, applying the triangle inequality, we can get
 \begin{equation}\label{3e37b}
 \begin{split}
   \|u^n-U^n\| + \|v^n-V^n\| \, \le& \, \| \zeta_{1}^n \| + \| \chi_{1}^n \| + \| \zeta_{2}^n \| + \| \chi_{2}^n \| \\
   \le& \, C \, \big( h^2 + \tau^2 \big).
 \end{split}
 \end{equation}
\\
In order to derive an error estimate in $H^1_0(\Omega)$ norm, we rewrite \eqref{3e23} as
\begin{equation}\label{3e38}
\begin{split}
&\big(\, ^C{D}^{\alpha}_{\tau} \chi_{1}^n, w \big) \, + \, M_1 \big(l(U^{n, \alpha}), l(V^{n, \alpha})\big) \, (\nabla \chi_{1}^{n, \alpha}, \nabla w) \qquad\\
&\;= \, \big(\, ^C{D}^{\alpha}_{\tau} R_hu^n - \, ^{C}{D}^{\alpha}_{t_{n-\frac{\alpha}{2}}}u , w \big) + M_1 \big(l(u^{n-\frac{\alpha}{2}}), l(v^{n-\frac{\alpha}{2}})\big) (\Delta u^{n, \alpha} - \Delta u^{n-\frac{\alpha}{2}}, w) \\
& \qquad + \Big\{ M_1 \big(l(U^{n, \alpha}), l(V^{n, \alpha})\big) - M_1 \big(l(u^{n-\frac{\alpha}{2}}), l(v^{n-\frac{\alpha}{2}})\big) \Big\} (\Delta u^{n, \alpha}, w) \\
& \qquad + \big( f_{1}(u^{n-\frac{\alpha}{2}}, v^{n-\frac{\alpha}{2}}) - f_{1}(U^{n, \alpha}, V^{n, \alpha}), w \big).
\end{split}
\end{equation}
We choose $w = \, ^C{D}^{\alpha}_{\tau} \chi_{1}^n$ in \eqref{3e38} to get
\begin{equation}\label{3e38a}
\begin{split}
&\big(\, ^C{D}^{\alpha}_{\tau} \chi_{1}^n, ^C{D}^{\alpha}_{\tau} \chi_{1}^n \big) \, + \, M_1 \big(l(U^{n, \alpha}), l(V^{n, \alpha})\big) \, \big(\nabla \chi_{1}^{n, \alpha}, \nabla ^C{D}^{\alpha}_{\tau} \chi_{1}^n\big) \qquad\\
&\;= \, \big(\, ^C{D}^{\alpha}_{\tau} R_hu^n - \, ^{C}{D}^{\alpha}_{t_{n-\frac{\alpha}{2}}}u , \, ^C{D}^{\alpha}_{\tau} \chi_{1}^n \big) + \big( f_{1}(u^{n-\frac{\alpha}{2}}, v^{n-\frac{\alpha}{2}}) - f_{1}(U^{n, \alpha}, V^{n, \alpha}), \, ^C{D}^{\alpha}_{\tau} \chi_{1}^n \big) \\
& \qquad + \Big\{ M_1 \big(l(U^{n, \alpha}), l(V^{n, \alpha})\big) - M_1 \big(l(u^{n-\frac{\alpha}{2}}), l(v^{n-\frac{\alpha}{2}})\big) \Big\} \big(\Delta u^{n, \alpha}, \, ^C{D}^{\alpha}_{\tau} \chi_{1}^n\big) \\
& \qquad + M_1 \big(l(u^{n-\frac{\alpha}{2}}), l(v^{n-\frac{\alpha}{2}})\big) \big(\Delta u^{n, \alpha} - \Delta u^{n-\frac{\alpha}{2}}, \,  ^C{D}^{\alpha}_{\tau} \chi_{1}^n\big).
\end{split}
\end{equation}
Now, dividing \eqref{3e38a} by $M_1 \big(l(U^{n, \alpha}), l(V^{n, \alpha})\big)$ and then using bound of $M_1$, Cauchy-Schwarz inequality, we can arrive at 
\begin{equation}\label{3e39}
\begin{split}
& \frac{1}{m_2} \, \big\| \, ^C{D}^{\alpha}_{\tau} \chi_{1}^n \big\|^2 + \big(\nabla \chi_{1}^{n, \alpha}, \nabla \, ^C{D}^{\alpha}_{\tau} \chi_{1}^n\big) \qquad\\
&\; \le \frac{1}{m_1} \, \big\| \, ^C{D}^{\alpha}_{\tau} R_hu^n - \, ^{C}{D}^{\alpha}_{t_{n-\frac{\alpha}{2}}}u \big\| \, \big\| \, ^C{D}^{\alpha}_{\tau} \chi_{1}^n \big\| + \frac{m_2}{m_1} \, \big\|\Delta u^{n, \alpha} - \Delta u^{n-\frac{\alpha}{2}}\big\| \, \big\| \, ^C{D}^{\alpha}_{\tau} \chi_{1}^n \big\| \\
& \qquad + \frac{1}{m_1} \, \big| M_1 \big(l(U^{n, \alpha}), l(V^{n, \alpha})\big) - M_1 \big(l(u^{n-\frac{\alpha}{2}}), l(v^{n-\frac{\alpha}{2}})\big) \big| \, \big\| \Delta u^{n, \alpha} \big\| \, \big\| \, ^C{D}^{\alpha}_{\tau} \chi_{1}^n \big\| \\
& \qquad + \frac{1}{m_1} \, \big\| f_{1}(u^{n-\frac{\alpha}{2}}, v^{n-\frac{\alpha}{2}}) - f_{1}(U^{n, \alpha},  V^{n, \alpha}) \big\| \, \big\| \, ^C{D}^{\alpha}_{\tau} \chi_{1}^n \big\|.
\end{split}
\end{equation}
Further for $a,b>0,$ using the inequality $ab\leq \frac{a^2}{2\epsilon}+ \frac{\epsilon b^2}{2},$ with $\epsilon =4 \, m_2$ and from \eqref{3e27}, we have
\begin{equation}\label{3e40}
\begin{split}
& \frac{1}{m_2} \, \big\| \, ^C{D}^{\alpha}_{\tau} \chi_{1}^n \big\|^2 + \big(\nabla \chi_{1}^{n, \alpha}, \nabla \, ^C{D}^{\alpha}_{\tau} \chi_{1}^n\big) \qquad\\
&\; \le \frac{2 \, m_2}{m_1^2} \, \big\| \, ^C{D}^{\alpha}_{\tau} R_hu^n - \, ^{C}{D}^{\alpha}_{t_{n-\frac{\alpha}{2}}}u \big\|^2 + \frac{2 \, m_2^2}{m_1^2} \, \big\|\Delta u^{n, \alpha} - \Delta u^{n-\frac{\alpha}{2}}\big\|^2 \\
& \qquad + \frac{2 R_u^2\, m_2}{m_1^2} \, \big| M_1 \big(l(U^{n, \alpha}), l(V^{n, \alpha})\big) - M_1 \big(l(u^{n-\frac{\alpha}{2}}), l(v^{n-\frac{\alpha}{2}})\big) \big|^2  \\
& \qquad + \frac{2 \, m_2}{m_1^2} \, \big\| f_{1}(u^{n-\frac{\alpha}{2}}, v^{n-\frac{\alpha}{2}}) - f_{1}(U^{n, \alpha},  V^{n, \alpha}) \big\|^2 + \frac{1}{2 \, m_2} \, \big\| \, ^C{D}^{\alpha}_{\tau} \chi_{1}^n \big\|^2.
\end{split}
\end{equation}
Using \eqref{3e28} and \eqref{3e29} along with Poincar\'e inequality in \eqref{3e40}, we get 
\begin{equation}\label{3e41}
\begin{split}
\big(\nabla \chi_{1}^{n, \alpha}, \nabla \, ^C{D}^{\alpha}_{\tau} \chi_{1}^n\big) \; \le   C \, \Big\{ \big\| \, ^C{D}^{\alpha}_{\tau} R_hu^n - \, ^{C}{D}^{\alpha}_{t_{n-\frac{\alpha}{2}}}u \big\|^2 + \big\|\Delta u^{n, \alpha} - \Delta u^{n-\frac{\alpha}{2}}\big\|^2& \\
  + \| \nabla \zeta_{1}^{n, \alpha}\|^2 + \| \nabla \chi_{1}^{n, \alpha}\|^2 + \| \nabla \zeta_{2}^{n, \alpha}\|^2 + \| \nabla \chi_{2}^{n, \alpha}\|^2& \\
+ \| u^{n, \alpha} - u^{n-\frac{\alpha}{2}}\|^2 + \|v^{n, \alpha} - v^{n-\frac{\alpha}{2}}\|^2& \Big\},
\end{split}
\end{equation}
where $C$ is dependent on $m_1,$ $m_2,$ $L'_1,$ $K'_1,$ $L_1,$ $K_1,$ $R_u.$\\
Also, from Taylor's theorem, we have
\begin{equation}\label{3e42}
\begin{split}
\|\Delta u^{n, \alpha} - \Delta u^{n-\frac{\alpha}{2}}\| \, \le \, C \, \tau^2.
\end{split}
\end{equation}
Using Lemma \ref{3l1}, \eqref{3e22}, \eqref{3e31}-\eqref{3e33} and \eqref{3e42} in \eqref{3e41}, we can obtain
\begin{equation}\label{3e43}
\begin{split}
^C{D}^{\alpha}_{\tau} \| \nabla \chi_{1}^n \|^2 \le& \, C \big(\| \nabla \chi_{1}^{n, \alpha}\|^2 + \| \nabla \chi_{2}^{n, \alpha}\|^2 + \big( h + \tau^2 \big)^2 \big).
\end{split}
\end{equation}
Similarly, one can obtain an estimate for $\| \nabla \chi_{2}^{n, \alpha}\|$ as follows:
\begin{equation}\label{3e44}
\begin{split}
^C{D}^{\alpha}_{\tau} \| \nabla \chi_{2}^n \|^2 \le& \, C \big(\| \nabla \chi_{1}^{n, \alpha}\|^2 + \| \nabla \chi_{2}^{n, \alpha}\|^2 + \big( h + \tau^2 \big)^2 \big).
\end{split}
\end{equation}
Adding \eqref{3e43} and \eqref{3e44}, we get 
\begin{equation}\label{3e45}
\begin{split}
^C{D}^{\alpha}_{\tau} \big( \| \nabla \chi_{1}^n \|^2 + \| \nabla \chi_{2}^n \|^2 \big) \le& \, C \, \big( h + \tau^2 \big)^2 + C \Big( 1 - \frac{\alpha}{2}\Big)^2 \big(\| \nabla \chi_{1}^n \|^2 + \| \nabla \chi_{2}^n \|^2 \big) \\
&+ \frac{C \, \alpha^2}{4} \big(\| \nabla \chi_{1}^{n-1} \|^2 + \| \nabla \chi_{2}^{n-1} \|^2 \big). \\
\end{split}
\end{equation}
Applying Lemma $\ref{Dong}$ $\Big(\mbox{with} \; \lambda_1 = C \, \big( 1 - \frac{\alpha}{2}\big)^2,$ $\lambda_2 = \frac{C \, \alpha^2}{4},$ $\phi^n = C \, ( h + \tau^2 )^2\Big)$, we can get
\begin{equation}\label{3e46}
\begin{split}
\| \nabla \chi_{1}^n \|^2 + \| \nabla \chi_{2}^n \|^2  \le& \, C \, \big( h + \tau^2 \big)^2,
\end{split}
\end{equation}
where $C$ is dependent on $\lambda_1, \lambda_2, \alpha, T.$\\
Finally, we use the Triangle inequality together with the estimates \eqref{3e46}, \eqref{3e22} to obtain
\begin{equation}\label{3e47b}
\begin{split}
\|\nabla u^n - \nabla U^n\| + \| \nabla v^n - \nabla V^n\| \, \le& \, \| \nabla \zeta_{1}^n \| + \| \nabla \chi_{1}^n \| + \| \nabla \zeta_{2}^n \| + \| \nabla \chi_{2}^n \| \\
\le& \, C \, \big( h + \tau^2 \big).
\end{split}
\end{equation} 
This completes the proof.
\end{proof}
\begin{remark}
		In present work, the convergence of the proposed scheme is proved without considering the weak singularity of the solution. We leave the convergence analysis of weak singularity case as a future work.
\end{remark}
\section{Numerical experiments}
In this section, we perform some numerical experiments by considering two different problems with known exact solutions. In both problems, we take the final time $T=1$ and tolerance $\epsilon=10^{-7}$ for stopping Newton's iteration. We denote the number of sub-intervals in time by $N$. Moreover, let $(M_s+1)$ be the number of node points in each spatial direction. In order to obtain the order of convergence in spatial direction in $L^2(\Omega)$ and $H^1_0(\Omega)$ norms, we take $N= M_s$ for different values of $M_s$. Similarly, to calculate the convergence rate in temporal direction in $L^2(\Omega)$ norm, we take $ M_s=N$ for different values of $N$.\\
\begin{example}\label{3E1}
For our first example, in \eqref{3e4}, we consider the spatial domain $\Omega = (0, 1)$, $M_1(z,w)=3+\sin z + \cos w,$ $M_2(z,w)=5+\cos z + \sin w.$ The functions $f_1$ and $f_2$ are chosen such that the analytical solutions of the equation \eqref{3e4} be  $u(x,t)= t^{2+ \alpha } \sin 2 \pi x$ and $v(x,t)= t^{3- \alpha } \sin \pi x$. 	
\end{example}
Error and convergence rate in the spatial direction in $L^2(\Omega)$ and $H^1_0(\Omega)$ norms are given in Tables \ref{3t1} and \ref{3t2}, respectively. Furthermore, Table \ref{3t3} shows errors and convergence rates in the temporal direction in $L^2(\Omega)$ norm.
\begin{table}[h!]
	\centering
	\begin{tabular}{|c|c|c|c|c|c|}
		\hline
		& $M_s$  & \multicolumn{1}{c|}{$\|u^n-U^n\|$} & \multicolumn{1}{c|}{Rate} & \multicolumn{1}{c|}{$\|v^n-V^n\|$} & \multicolumn{1}{c|}{Rate} \\
		\hline
	    & $2^6$ &  7.17E-04   & 1.999438471  & 1.60E-04 & 2.000114366 \\
		\multicolumn{1}{|l|}{$\alpha=0.4$} & $2^7$ & 1.79E-04 & 1.999759291 & 3.99E-05 & 2.000074895 \\
		& $2^8$ & 4.48E-05  & 1.99989373  & 9.98E-06 & 2.000038509 \\
		& $2^9$ & 1.12E-05 & - &  2.49E-06 & - \\
		\hline
		& $2^6$ & 7.21E-04 & 1.999524396 &  1.58E-04 & 2.000010343 \\
		\multicolumn{1}{|l|}{$\alpha=0.7$} & $2^7$ & 1.80E-04 & 1.999830784  & 3.95E-05  & 2.00001483 \\
		& $2^8$ & 4.51E-05 & 1.999928113 &  9.89E-06  & 2.000000769 \\
		& $2^9$ & 1.13E-05 & - & 2.47E-06  & - \\
		\hline 				
	\end{tabular}
    \caption{(Example-\ref{3E1}) \emph{Error and convergence rate in spatial direction in $L^2(\Omega)$ norm.}}
	\label{3t1}
\end{table}
\vspace{15cm}
\begin{table}[h!]
	\centering
	\begin{tabular}{|c|c|c|c|c|c|}
		\hline
		& $M_s$  & \multicolumn{1}{c|}{$\|u^n-U^n\|$} & \multicolumn{1}{c|}{Rate} & \multicolumn{1}{c|}{$\|v^n-V^n\|$} & \multicolumn{1}{c|}{Rate} \\
		\hline
		& $2^6$ & 1.26E-01 & 0.999784124 & 3.15E-02 & 0.99994853 \\
		\multicolumn{1}{|l|}{$\alpha=0.4$} & $2^7$ & 6.30E-02 & 0.999946033 & 1.57E-02 & 0.999987136 \\
		& $2^8$ & 3.15E-02 & 0.999986509 & 7.87E-03 &  0.999996784 \\
		& $2^9$ & 1.57E-02 & - & 3.93E-03 & - \\
		\hline
		& $2^6$ & 1.26E-01 & 0.999784473 & 3.15E-02 & 0.999948918 \\
		\multicolumn{1}{|l|}{$\alpha=0.7$} & $2^7$ & 6.30E-02 & 0.99994612 & 1.57E-02 & 0.999987232 \\
		& $2^8$ & 3.15E-02 &  0.99998653 & 7.87E-03 & 0.999996808  \\
		& $2^9$ & 1.57E-02 &  - & 3.93E-03 & - \\
		\hline			
	\end{tabular}
	\caption{(Example-\ref{3E1}) \emph{Error and convergence rate in spatial direction in $H^1_0(\Omega)$ norm.}}
	\label{3t2}
\end{table}

\begin{table}[h!]
	\centering
	\begin{tabular}{|c|c|c|c|c|c|}
		\hline
		& $N$  & \multicolumn{1}{c|}{$\|u^n-U^n\|$} & \multicolumn{1}{c|}{Rate} & \multicolumn{1}{c|}{$\|v^n-V^n\|$} & \multicolumn{1}{c|}{Rate} \\
		\hline
		& $2^6$ & 7.17E-04 & 1.999438471 & 1.60E-04 & 2.000114366 \\
		\multicolumn{1}{|l|}{$\alpha=0.4$} & $2^7$ & 1.79E-04 & 1.999759291 & 3.99E-05 & 2.000074895 \\
		& $2^8$ & 4.48E-05 & 1.99989373 & 9.98E-06 & 2.000038509 \\
		& $2^9$ & 1.12E-05 & - & 2.49E-06 & - \\
		\hline
		& $2^6$ & 7.21E-04 & 1.999524396 & 1.58E-04 & 2.000010343 \\
		\multicolumn{1}{|l|}{$\alpha=0.7$} & $2^7$ & 1.80E-04 & 1.999830784 & 3.95E-05 & 2.00001483 \\
		& $2^8$ & 4.51E-05 & 1.999928113 & 9.89E-06 & 2.000000769 \\
		& $2^9$ & 1.13E-05 & - & 2.47E-06 & - \\
		\hline	
	\end{tabular}
	\caption{(Example-\ref{3E1}) \emph{Error and convergence rate in temporal direction in $L^2(\Omega)$ norm.}}
	\label{3t3}
\end{table}
\begin{example}\label{3E2}
		In this example, we take $\Omega = (0, 1) \times (0,1)$, $M_1(z,w)=3+\sin z + \cos w,$ $M_2(z,w)=5+\cos z + \sin w.$ We choose $f_1$ and $f_2$ such that the analytical solution of equation \eqref{3e4} be $u(x,y,t)= t^3 \sin 2 \pi x \, \sin 2 \pi y$ and $v(x,y,t)= t^4 \sin \pi x \, \sin \pi y$. 	
\end{example}
Error and convergence rate in the spatial direction in $L^2(\Omega)$ and $H^1_0(\Omega)$ norms are given in Tables \ref{3t4} and \ref{3t5}, respectively. Table \ref{3t6} shows errors and convergence rates in the temporal direction in $L^2(\Omega)$ norm. For $\alpha = 0.5,$ the graphs of exact and numerical solutions are shown in Figures \ref{3fig1} and \ref{3fig2}.   
\begin{table}[h!]
	\centering
	\begin{tabular}{|c|c|c|c|c|c|}
		\hline
		& $M_s$  & \multicolumn{1}{c|}{$\|u^n-U^n\|$} & \multicolumn{1}{c|}{Rate} & \multicolumn{1}{c|}{$\|v^n-V^n\|$} & \multicolumn{1}{c|}{Rate} \\
		\hline
		& $2^3$ & 8.76E-02 & 1.910229823  & 2.05E-02 & 1.983174711 \\
		\multicolumn{1}{|l|}{$\alpha=0.5$} & $2^4$ & 2.33E-02 & 1.976537254 & 5.18E-03 & 1.996676573 \\
		& $2^5$ & 5.92E-03 & 1.993876169  & 1.30E-03 & 1.999788928 \\
		& $2^6$ & 1.49E-03 & - & 3.24E-04 & - \\
		\hline
		& $2^3$ & 8.77E-02 & 1.910535849 & 1.99E-02 & 1.981811644 \\
		\multicolumn{1}{|l|}{$\alpha=0.9$} & $2^4$ & 2.33E-02 & 1.976942804 & 5.04E-03 & 1.995308014 \\
		& $2^5$ & 5.92E-03 & 1.994131588 & 1.26E-03 & 1.998953039 \\
		& $2^6$ & 1.49E-03 & - & 3.16E-04 & - \\
		\hline	
	\end{tabular}
	\caption{(Example-\ref{3E2}) \emph{Error and convergence rate in spatial direction in $L^2(\Omega)$ norm.}}
	\label{3t4}
\end{table}
\vspace{5cm}
\begin{table}[h!]
	\centering
	\begin{tabular}{|c|c|c|c|c|c|}
		\hline
		& $M_s$  & \multicolumn{1}{c|}{$\|u^n-U^n\|$} & \multicolumn{1}{c|}{Rate} & \multicolumn{1}{c|}{$\|v^n-V^n\|$} & \multicolumn{1}{c|}{Rate} \\
		\hline
		& $2^3$ & 1.28E+00 & 0.986482547 & 3.24E-01 & 0.99590991 \\
		\multicolumn{1}{|l|}{$\alpha=0.5$} & $2^4$ & 6.48E-01 & 0.996624338 & 1.62E-01 & 0.99896682 \\
		& $2^5$ & 3.25E-01 & 0.999158027  & 8.13E-02 & 0.999740748 \\
		& $2^6$ & 1.63E-01 & - & 4.06E-02 & - \\
		\hline
		& $2^3$ & 1.28E+00 & 0.986646867 & 3.24E-01 & 0.995636078 \\
		\multicolumn{1}{|l|}{$\alpha=0.9$} & $2^4$ & 6.48E-01 & 0.996659575 & 1.62E-01 & 0.998904181 \\
		& $2^5$ & 3.25E-01 & 0.999164853 & 8.13E-02 & 0.999725974 \\
		& $2^6$ & 1.63E-01 & - & 4.06E-02 & - \\
		\hline 			
	\end{tabular}
	\caption{(Example-\ref{3E2}) \emph{Error and convergence rate in spatial direction in $H^1_0(\Omega)$ norm.}}
	\label{3t5}
\end{table}
\begin{table}[h!]
	\centering
	\begin{tabular}{|c|c|c|c|c|c|}
		\hline
		& $N$  & \multicolumn{1}{c|}{$\|u^n-U^n\|$} & \multicolumn{1}{c|}{Rate} & \multicolumn{1}{c|}{$\|v^n-V^n\|$} & \multicolumn{1}{c|}{Rate} \\
		\hline
		& $2^3$ & 8.76E-02 &  1.910229823 & 2.05E-02 & 1.983174711 \\
		\multicolumn{1}{|l|}{$\alpha=0.5$} & $2^4$ & 2.33E-02 & 1.976537254 & 5.18E-03 & 1.996676573 \\
		& $2^5$ & 5.92E-03 & 1.993876169 & 1.30E-03 & 1.999788928 \\
		& $2^6$ & 1.49E-03 & - & 3.24E-04 & - \\
		\hline
		& $2^3$ & 8.77E-02 & 1.910535849 & 1.99E-02 & 1.981811644 \\
		\multicolumn{1}{|l|}{$\alpha=0.9$} & $2^4$ & 2.33E-02 & 1.976942804 & 5.04E-03 & 1.995308014 \\
		& $2^5$ & 5.92E-03 & 1.994131588 & 1.26E-03 & 1.998953039 \\
		& $2^6$ & 1.49E-03 & - & 3.16E-04 & - \\
		\hline 		
	\end{tabular}
	\caption{(Example-\ref{3E2}) \emph{Error and convergence rate in temporal direction in $L^2(\Omega)$ norm.}}
	\label{3t6}
\end{table}
\begin{figure}[h!]
	\begin{center}
		\includegraphics[width=15cm]{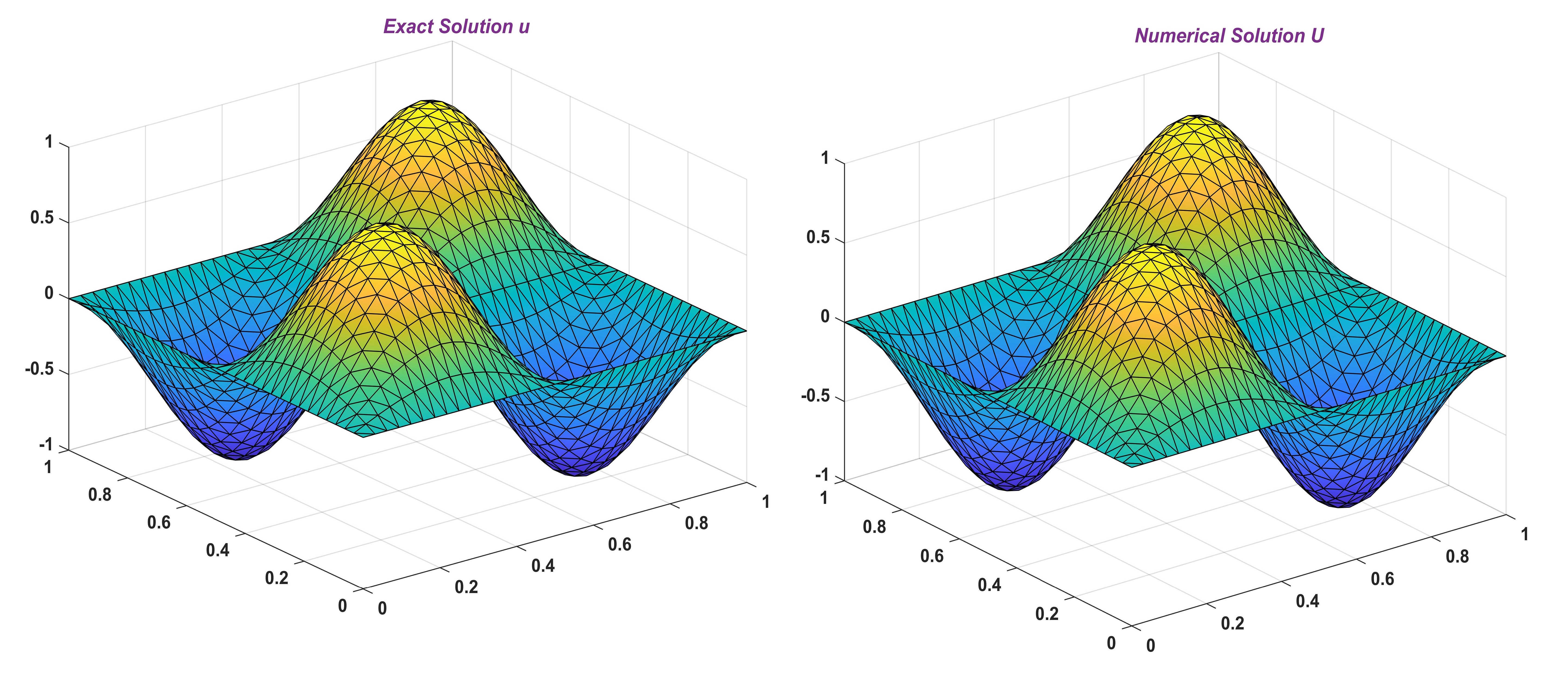}
	\end{center}
	\caption{(Example-\ref{3E2}) \emph{The exact solution $u$ and numerical solution $U$ for  $\alpha =0.5$.}}
	\label{3fig1}
\end{figure}
\vspace{5cm}
\begin{figure}[h!]
	\begin{center}
		\includegraphics[width=15cm]{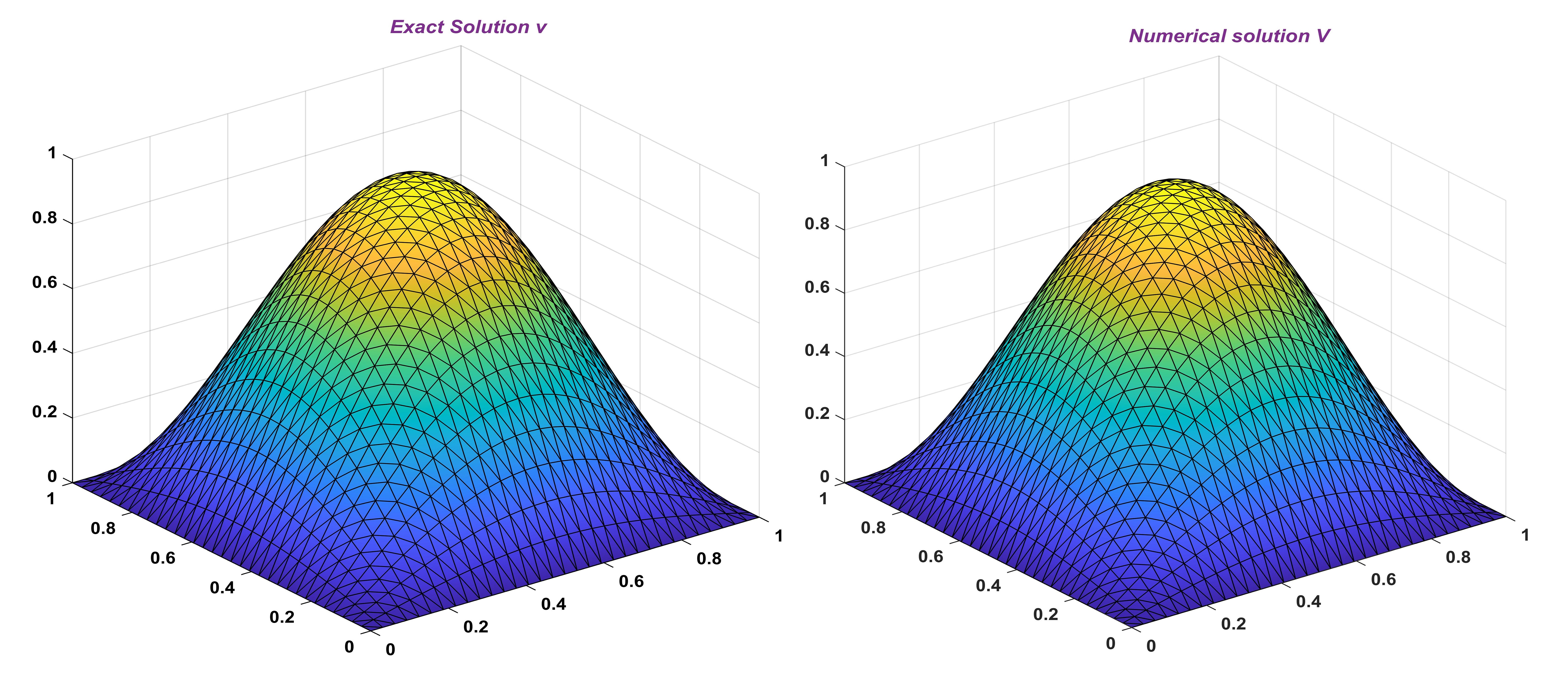}
	\end{center}
	\caption{(Example-\ref{3E2}) \emph{The exact solution $v$ and numerical solution $V$ for $\alpha =0.5$.}}
	\label{3fig2}
\end{figure}
\newpage
\noindent \textbf{Declarations:}\\
\textbf{Conflict of interest-} The author declares no competing interests.

	

\begin{thebibliography}{20}
  \bibitem{[me]} S. Chaudhary, \emph{Finite element analysis of nonlocal coupled  parabolic problem using Newton's method}, Comput. Math. Appl. 75-3 (2018), 981-1003.
	
  \bibitem{[sk]} S. Chaudhary, V. Srivastava, V. V. K. Srinivas Kumar, B. Srinivasan, \emph{Finite element approximation of nonlocal parabolic problem}, Numer. Methods Partial Differ. Eq., 33 (2017) 786-313.
 
  \bibitem{[CNS]} S. Chaudhary, \emph{Crank-Nicolson-Galerkin finite element scheme for nonlocal coupled parabolic problem using the Newton method}, Math. Meth. Appl. Sci., 41-2 (2018), 724-749.
  
  \bibitem{[sp1]} S. Chaudhary, P. J. Kundaliya, \emph{L1 scheme on graded mesh for subdiffusion equation with nonlocal diffusion term}, Math. Comput. Simul., 195 (2022), 119-137. 
  
  \bibitem{[CNCD]} D. Kumar, S. Chaudhary, V. V. K. Srinivas Kumar, \emph{Galerkin finite element schemes with fractional Crank-Nicolson method for the coupled time-fractional nonlinear diffusion system}, Comput. Appl. Math., 38, Article number: 123 (2019).
  	
  \bibitem{[Podlubny]} I. Podlubny, \emph{Fractional differential equations}, Academic press, (1999).
  	
  \bibitem{[Anatoly]} A. A. Kilbas, H. M. Srivastava, J. J. Trujillo, \emph{Theory and applications of fractional differential equations}, Elsevier, (2006).
  
  \bibitem{[Bruce]} B. J. West, \emph{Fractional calculus in bioengineering},  J. Stat. Phys., 126 (2007):1285-1286.
  
  \bibitem{[r5new]} Y. Luchko, M. Rivero, J. Trujillo, M. Pilar Velasco, \emph{Fractional models, non-locality, and complex systems}, Comput. Math. Appl., 59-3 (2010), 1048-1056.
  
  \bibitem{[ch]} M. Chipot, B. Lovat, \emph{On the asymptotic behaviour of some nonlocal problems}, Positivity 3 (1999), 65-81.
	
  \bibitem{[r6new]} S. B. Menezes, \emph{Remarks on weak solutions for a nonlocal parabolic problem}, Int. J. Math. Math. Sci., 2006 (2006), 1-10.
  
  \bibitem{[r1]} K. Diethelm, \emph{The Analysis of Fractional Differential Equations: An Application-Oriented Using Differential Operators of Caputo Type}, Lecture Notes in Mathematics, Springer, (2010).
 	
  \bibitem{[Alm]} R. M. P. Almeida, S. N. Antontsev, J. C. M. Duque, J. Ferreira, \emph{A reaction-diffusion model for the non-local coupled system: existence, uniqueness, long-time behaviour and localization properties of solutions}, IMA J. Appl. Math., 81-2 (2016), 344-364.

  \bibitem{[JCMD]} J. C. M. Duque, R. M. P. Almeida, S. N. Antontsev,  J. Ferreira, \emph{The Euler-Galerkin finite element method for a nonlocal coupled system of reaction-diffusion type}, J. Comput. Appl. Math., 296 (2016), 116-126.

 \bibitem{[C.A.]} C. A. Raposo, M. Sep{\'u}lveda, O. V. Villagr{\'a}n, D. C. Pereira,  M. L. Santos, \emph{Solution and asymptotic behavior for a nonlocal coupled system of reaction-diffusion}, Acta Appl. Math., 102 (2008), 37-56.

 \bibitem{[LI]} L. Li, L. Jin, S. Fang, \emph{Existence and uniqueness of the solution to a coupled fractional diffusion system}, Adv. Differ. Equ., 2015, Article number: 370 (2015).

 \bibitem{[cnm]}  B. Jin, B. Li, and Z. Zhou,  \emph{An analysis of the Crank-Nicolson method for subdiffusion}, IMA J. Numer. Anal., 38-1 (2017), 518-541.
 
 \bibitem{[yd]}  Y. Dimitrov,  \emph{Numerical approximations for fractional differential equations}, J. fractional calc. \& appl., 5-22 (2014), 1-45.
 
 \bibitem{[scs]}  Guang Hua Gao, Hai Wei Sun, Zhi Zhong Sun,  \emph{Stability and convergence of finite difference schemes for a class of time-fractional sub-diffusion equations based on certain superconvergence}, J. Comput. Phys., 280 (2015), 510-528.
  
 \bibitem{[CND]} D. Kumar, S. Chaudhary, V. V. K. Srinivas Kumar, \emph{Fractional Crank-Nicolson-Galerkin finite element scheme for the time-fractional nonlinear diffusion equation}, Numer. Methods Partial Differ. Eq., 35 (2019), 2056-2075.
 
 \bibitem{[r3]} T. Gudi, \emph{Finite element method for a nonlocal problem of Kirchhoff type}, SIAM J. Numer. Anal., 50-2 (2012), 657-668.
 
 \bibitem{[mn3]} J. Manimaran, L. Shangerganesh, \emph{Error estimates for Galerkin finite element approximations of time-fractional nonlocal diffusion equation},  Int. J. Comput. Math., 98-7 (2020), 1365-1384.
 
  \bibitem{[vth]} V. Thom$\acute{e}$e, \emph{Galerkin Finite Element Methods for Parabolic Problems}, Second revised and expanded ed., Springer, Berlin, (2006).
  
  \bibitem{[hr12]} C. Huang, M. Stynes, \emph{Optimal spatial $H^1-norm$ analysis of a finite element method for a time-fractional diffusion equation},  J. Comput. Appl. Math., 367 (2020), 112435.
  
  \bibitem{[r02]} M. Stynes, E. O'Riordan, J. Gracia, \emph{Error analysis of a finite difference method on graded meshes for a time-fractional diffusion equation}, SIAM J. Numer. Anal., 55 (2017), 1057-1079.
  
  \bibitem{[AAl2]} A. Alikhanov, \emph{A new difference scheme for the time fractional diffusion equation},  J. Comput. Phys., 280 (2015) 424-438.
  
  \bibitem{[r11]} N. Kopteva, \emph{Error analysis of the L1 method on graded and uniform meshes for a fractional derivative problem in two and three dimensions}, Math. Comp., 8 (2019) 2135-2155.
  
  \bibitem{[r8]} B. Jin, B. Li, Z. Zhou, \emph{Numerical analysis of nonlinear subdiffusion equations}, SIAM J. Numer. Anal., 56(1) (2018) 1-23.
  
  \bibitem{[r5a]} J. Ren, H. Liao, J. Zhang, Z. Zhang,  \emph{Sharp H1-norm error estimates of two time-stepping schemes for reaction-subdiffusion problems}, J. Comput. Appl. Math., 389 (2021), 113352.
  
  \bibitem{[mk]}  M. Al-Maskari, S. Karaa, \emph{The time-fractional Cahn-Hilliard equation: analysis and approximation}, IMA J. Numer. Anal., 2021. Doi:10.1093/imanum/drab025.
  
  \bibitem{[cl1]} C. Lubich, \emph{Discretized fractional calculus}, SIAM J. Math. Anal., 17 (1986), 704-719.
  
  \bibitem{[cl2]} C. Lubich, \emph{Convolution quadrature and discretized operational calculus. I}, Numer. Math., 52 (1988), 129-145.
  
  
\end{thebibliography}
\end{document}